\documentclass[12pt]{article}

\usepackage[english]{babel}
\usepackage{amssymb}
\usepackage{graphics}

\newcommand{\qed}{$\;\;\;\Box$}
\newenvironment{proof}{\par\smallbreak{\sl Proof.~}}
{\unskip\nobreak\hfill \qed \par\medbreak}
\newcounter{claim}
\renewcommand{\theclaim}{\arabic{claim}}
{\par\medskip\par}

{\qed\par\smallbreak}
\newcommand{\hide}[1]{}

\newcommand{\D}{{\cal D}}

\newcommand{\R}{{\mathbb R}}

\newcommand{\LL}{{\cal L}}

\newcommand{\beq}{\begin{equation}}
\newcommand{\ee}{\end{equation}}

\renewcommand{\d}{\partial}

\newtheorem{thm}{Theorem}[section]
\newtheorem{lemma}[thm]{Lemma}

\newtheorem{defn}[thm]{Definition}
\newtheorem{cor}[thm]{Corollary}
\newtheorem{rem}[thm]{Remark}
\newtheorem{ex}[thm]{Example}

\newcommand{\al}{\alpha}

\newcommand{\eps}{\varepsilon}

\newcommand{\om}{\omega}

\newcommand{\reff}[1]{(\ref{#1})}

\date{}
\title{
Exponential Dichotomy for Hyperbolic Systems with Periodic Boundary Conditions
}

\newcounter{thesame}
\author{
 R. Klyuchnyk \ \ \ I. Kmit
\ \ \ L. Recke\\ [5mm]
{\small
Institute for Applied Problems of Mechanics and Mathematics,
}
\\
{\small
Ukrainian Academy of Sciences}
\\
{\small   E-mail:
{\tt roman.klyuchnyk@gmail.com}}
\\ [5mm]
{\small
Institute of Mathematics, Humboldt University of Berlin, Germany}
\\
{\small and Institute for Applied Problems of Mechanics and Mathematics,
}
\\
{\small
Ukrainian Academy of Sciences}\\
{\small   E-mail:
{\tt kmit@mathematik.hu-berlin.de}}\\ [5mm]
{\small
Institute of Mathematics, Humboldt University of Berlin, Germany}
\\
{\small   E-mail:
{\tt recke@mathematik.hu-berlin.de}}
}

\begin{document}
\maketitle

\begin{abstract}
\noindent 
We investigate evolution families generated by general linear first-order hyperbolic 
systems  in one space dimension with periodic boundary conditions. 
We state explicit conditions 
on the coefficient functions that are sufficient for the existence of exponential
dichotomies on $\R$ in the space of continuous periodic functions. 
\end{abstract}


\section{Introduction}\label{sec:intr} 

\renewcommand{\theequation}{{\thesection}.\arabic{equation}}
\setcounter{equation}{0}

\subsection{Problem setting and the main results}\label{sec:setting}

We consider linear homogeneous first-order hyperbolic systems in one space variable
  \begin{equation}
 \partial_{t}u_j
 +a_j(x,t)\partial_{x}u_j
 +\sum_{k=1}^{n}b_{jk}(x,t)u_k
 =0,
 \;\;\; (x,t)\in\R^2,\;\;\; j\le n,
 \label{f1l}
 \end{equation}
 with periodic boundary conditions
 \beq\label{f2l}
  u_j(x+1,t)=u_j(x,t),\;\;\;(x,t)\in\R^2,\;\;\; j\le n,
 \ee
and initial conditions 
 \beq\label{f3l}
  u_j(x,s)=u_j^{s}(x),\;\;\;x\in\R,\;\;\; j\le n.
 \ee
Here $s\in\R$ is an arbitrary fixed initial time.
Throughout the paper, we suppose that the coefficient functions 
 $a_j, b_{jk}:\R^2\to\R$  are bounded, continuous, and
 $1$-periodic 
with respect to the space variable $x$. Furthermore, the initial data $u^{s}_j:\R\to\R$ are supposed to be continuous and $1$-periodic.
Finally, we suppose that the leading order coefficients $a_j$ have bounded and continuous
partial derivatives in $x$ and $t$, and that the following condition is fulfilled:
\beq
\label{f5l}
\inf\{|a_j(x,t)|:\, (x,t) \in \R^2, j\le n\}>0.
\ee

Our goal is to state conditions on the coefficients $a_j$ and $b_{jk}$ such that
the evolution family
generated by the initial-boundary value problem \reff{f1l}--\reff{f3l} has an 
exponential dichotomy on $\R$. 
In particular, we will prove that the following conditions are sufficient for the existence
of an exponential dichotomy:
\begin{itemize}
\item
 $b_{jj}(x,t)\not=0$ for all $j$, $x$ and $t$, while
for all $j\not=k$ the functions $|b_{jk}|$ are uniformly small (in terms of  
 the coefficients $a_j$ and $b_{jj}$, see 
Theorem \ref{thm:th3}).
\item
 $a_j(x,t)b_{jj}(x,t)<0$ for all $j$, $x$ and $t$, while for all  $j\not=k$, $x$, $t$ it holds $a_j(x,t) \not=a_k(x,t)$ and $b_{jk}(x,t)\to 0$ as $t \to \pm\infty$ uniformly in $x$ (cf.\ Theorem \ref{thm:th2}).
\end{itemize}

To formulate our results more precisely, let us introduce the evolution family
generated by \reff{f1l}--\reff{f3l}, whose existence is stated in Theorem \ref{thm:th1}
below. Recall that,
for $j\le n$ and $(x,t)\in\R^2$, the $j$-th characteristic of the system \reff{f1l} 
through the point $(x,t)$ is defined as the solution 
$\tau_j(\xi,x,t)$ of the initial value problem
$$
\d_{\tau}\xi_{j}(\tau,x,t)=a_j(\xi_j(\tau,x,t),\tau), \;\;\; \xi_{j}(t,x,t)=x.
$$
It is easy to show by  integration along characteristics
that, if $u=(u_1,...,u_n)$ is a classical solution to \reff{f1l}--\reff{f3l}, then 
\begin{eqnarray}\label{f8l}
&&u_j(x,t)=\displaystyle \exp\left(-\int_{s}^{t}b_{jj}(\xi_j(r,x,t),r)d r\right)
u_{j}^{s}(\xi_{j}(s,x,t))\nonumber\\
&&-\displaystyle\int_{s}^{t}\exp\left(-\int_{\tau}^{t}b_{jj}(\xi_j(r,x,t),r)dr\right)
\sum_{k\neq j}b_{jk}(\xi_j(\tau,x,t),\tau)u_{k}(\xi_{j}(\tau,x,t),\tau)d\tau
\end{eqnarray}
for all $(x,t)\in\R^2$ and all $j\le n$.
Vice versa, if the initial functions $u_j^s$ are $C^1$-smooth, then any solution $u$ to \reff{f8l}, 
which is $1$-periodic in $x$, is a classical solution to \reff{f1l}--\reff{f3l}.

We will work in the Banach  space
$$
C_{per}(\R;\R^n):=\{u\in C(\R;\R^n): u(x+1)=u(x) \mbox{ for all } x\in\R\} 
$$
normed by
\begin{equation}
\label{norm}
\|u\|:=\sup\{|u_j(x)|: x\in\R,\; j\le n\}.
\end{equation}
As usual, the space of all linear bounded operators $A: C_{per}(\R;\R^n)\to C_{per}(\R;\R^n)$ will 
be denoted by $\mathcal{L}(C_{per}(\R;\R^{n}))$, and the operator norm will be defined by
$$\|A\|:=\sup\{ \|Au\|: u\in C_{per}(\R;\R^n), \|u\|\le 1\}.
$$

The following result states that the initial-boundary value problem \reff{f1l}--\reff{f3l}
is well-posed.

\begin{thm}\label{thm:th1}
Suppose \reff{f5l} holds. Then, given $s\in\R$, for every $u^{s}\in C_{per}(\R;\R^n)$ there exists exactly one continuous 
function $u:\R^2\to\R^n$ satisfying 
 \reff{f2l} and \reff{f8l}. Moreover, the map
$$
u^{s}\mapsto U(t,s)u^{s}:=u(\cdot,t)
$$
from $C_{per}(\R;\R^n)$ to itself
defines  a strongly continuous, exponentially bounded evolution family of invertible operators $U(t,s)\in \LL(C_{per}(\R;\R^n))$, which means that
\begin{itemize}
\item
$U(t,t)=I$ and $U(t,s)=U(t,r)U(r,s)$  for all $t, r, s \in \R,$
\item
 the map $(t,s)\in \R^2\mapsto U(t,s)u\in C_{per}(\R;\R^n)$ is continuous for each
$u\in C_{per}(\R;\R^n),$
\item
  there exist $K \ge 1$ and $\omega \in \R$ such that
\beq\label{k3}
\|U(t,s)\|\le Ke^{\omega(t-s)} \mbox{ for all } t\ge s.
\ee
\end{itemize}
\end{thm}

In order to formulate our main results,
let us introduce the following notation:
\begin{eqnarray*}
\al_j^-&:=&\inf\left\{a_j(x,t):\, (x,t) \in \R^2\right\},\\
\al_j^+&:=&\sup\left\{a_j(x,t):\, (x,t) \in \R^2\right\},\\
\beta_j^-&:=&\inf\left\{b_{jj}(x,t):\, (x,t) \in \R^2\right\},\\
\beta_j^+&:=&\sup\left\{b_{jj}(x,t):\, (x,t) \in \R^2\right\},\\
\label{betadef}
\beta_j&:=&\sup\left\{\sum_{k\not=j}|b_{jk}(x,t)|:\, (x,t) \in \R^2\right\}.
\end{eqnarray*}
By the assumption \reff{f5l},
either $\al_j^->0$ or $\al_j^+<0$.

We are now prepared to formulate our first sufficient condition
for the existence of an exponential dichotomy for \reff{f1l}--\reff{f3l} on $\R$.

\begin{thm}\label{thm:th3}
Suppose \reff{f5l} holds. Moreover, suppose that 
\beq
\label{betacond}
\inf\{|b_{jj}(x,t)|:\, (x,t) \in \R^2, j\le n\}>0
\ee 
and  the following inequalities are true for all $j\le n$:
\begin{eqnarray}
\label{a--}
\beta_j< \beta_j^- \frac{\al_j^-}{\al_j^+}&\mbox{ if } & \al_j^->0, \;\beta_j^->0,\\
\label{a++}
\beta_j < -\beta_j^+ \frac{\al_j^+\left(1-e^{-\beta_j^+/\al_j^-}\right)}{\al_j^-\left(1-e^{-\beta_j^+/\al_j^+}\right)}   &\mbox{ if } & \al_j^+<0, \;\beta_j^+<0,\\
\label{a-+}
\beta_j < \beta_j^- \frac{1-e^{\beta_j^+/\al_j^+}}{1-e^{-\beta_j^-/\al_j^-}}\left(e^{\beta_j^+/\al_j^+-\beta_j^-/\al_j^-}-e^{\beta_j^+/\al_j^+}+1\right)^{-1}
&\mbox{ if } & \al_j^->0, \;\beta_j^+<0,\\
\label{a+-}
\beta_j <  -\beta_j^+ \frac{1-e^{\beta_j^-/\al_j^-}}{1-e^{-\beta_j^+/\al_j^+}}\left(e^{\beta_j^-/\al_j^--\beta_j^+/\al_j^+}-e^{\beta_j^-/\al_j^-}+1\right)^{-1}
&\mbox{ if } & \al_j^+<0, \;\beta_j^->0.
\end{eqnarray}
Then the evolution family $U(t,s)$ has an exponential dichotomy on $\R$, 
which means that there exist a projection $P=P^2 \in \mathcal{L}(C_{per}(\R;\R^n))$ and positive constants $M$ and $\omega$ such that 
$$
\|U(t,0)PU(0,s)\|+
\|U(s,0)(I-P)U(0,t)\|\le Me^{-\omega(t-s)} \mbox{ for all } t\ge s.
$$ 
\end{thm}

\begin{rem}\rm
Roughly speaking, Theorem \ref{thm:th3} claims the following: 
Given  $a_j$ and $b_{jj}$ satisfying \reff{f5l} and \reff{betacond}, 
the evolution family $U(t,s)$ has an exponential dichotomy on $\R$ if
  $b_{jk}$ with $j \not=k$ are  sufficiently small, in the sense of  the 
inequalities \reff{a--}--\reff{a+-}.
\end{rem}

\begin{rem}\label{r1}\rm
If the coefficients  $a_j$ and $b_{jj}$ are constants, then 
$$
\al_j^+=\al_j^-=a_j,\;
\beta_j^+=\beta_j^-=b_{jj}
$$
and, hence, \reff{betacond}--\reff{a+-} is equivalent to
\begin{eqnarray*}
\beta_j < |b_{jj}| &\mbox{if} & a_jb_{jj}>0,\\
\beta_j <  \frac{|b_{jj}|}{2e^{-b_{jj}/a_j}-1}
&\mbox{if} & a_jb_{jj}<0.
\end{eqnarray*}
In particular, if $n=2$, then these conditions read 
\begin{equation}
\label{n=2}
\begin{array}{rl}
|b_{12}| < |b_{11}| \mbox{ if }  a_1b_{11}>0,&
|b_{21}| < |b_{22}| \mbox{ if }  a_2b_{22}>0,\\
|b_{12}| <  \frac{|b_{11}|}{2e^{-b_{11}/a_1}-1}
\mbox{ if }  a_1b_{11}<0, &\displaystyle |b_{21}| <  \frac{|b_{22}|}{2e^{-b_{22}/a_2}-1}
\mbox{ if }  a_2b_{22}<0.
\end{array}
\end{equation}
\end{rem}
\vskip0.8mm

Now we formulate our second sufficient condition
for the existence of an exponential dichotomy for \reff{f1l}--\reff{f3l} on $\R$ .

\begin{thm}\label{thm:th2}
Suppose that \reff{f5l} is true and that either
\begin{equation}\label{k0}
\sup\left\{b_{jj}(x,t):\, (x,t) \in \R^2, j\le n\right\}<0
\end{equation}
or 
\begin{equation}\label{k00}
\inf\left\{b_{jj}(x,t):\, (x,t) \in \R^2, j\le n\right\}>0.
\end{equation}
Moreover, suppose the following:
\begin{equation}\label{f143l}
\begin{array}{ll}
\mbox{for all } 1\le j\neq k\le n \mbox{ and } \varepsilon>0 \mbox{ there exists } c>0
\mbox{ such that }
\\
|b_{jk}(x,t)|<\varepsilon \mbox{ for all } x\in \R \mbox{ and } t\in \R\setminus[-c,c]
\end{array}
\end{equation}
and
\beq\label{f9l}
\begin{array}{ll}
\mbox{for all } 1\le j\neq k\le n \mbox{ there exists }  \tilde b_{jk}\in C^1(\R^2) \mbox{ such that}\\
b_{jk}(x,t)=\tilde{b}_{jk}(x,t)(a_j(x,t)-a_k(x,t)) \mbox{ for all } (x,t)\in\R^2.
\end{array}
\ee 
Then the evolution family $U(t,s)$ has an exponential dichotomy on $\R$.
\end{thm}

\begin{rem}\rm
The condition \reff{f143l} implies that $b_{jk}(x,t)\to0$ as $t\to\pm\infty$ uniformly in $x\in\R$
for each $j\neq k$. 
In particular, \reff{f143l} is satisfied for time-constant or time-periodic "non-diagonal" coefficients if and only if they are identically zero. This is a disadvantage of Theorem \ref{thm:th2}. 
However, the advantage of Theorem \ref{thm:th2} is that the "non-diagonal" 
coefficients  have to be small only for large $|t|$ rather than uniformly over $x$ and $t$ (like in Theorem \ref{thm:th3}).
\end{rem}

Our approach to proving  Theorems \ref{thm:th3} and \ref{thm:th2}
is based on  the following criterion \cite[Theorem 1.1]{Latn}:
\begin{thm} \label{thm:thD}
A strongly continuous, exponentially bounded evolution family $\{{U}(t,s)\}_{t\ge s}$ on a Banach space $X$ has an exponential 
dichotomy on $\R$ if and only if for every bounded and continuous map $\tilde{f}:\R\to X$ there exists a unique bounded and continuous 
map $\tilde{u}:\R\to X$ such that
\beq\label{mild}
\tilde{u}(t)={U}(t,s)u(s)+\int_{s}^{t}U(t,\tau)\tilde{f}(\tau)d\tau \;\;\;\mbox{for all}\;\;\; t\geq s.
\ee 
\end{thm}

For proving  Theorems \ref{thm:th3} and \ref{thm:th2}, we 
set $X=C_{per}(\R;\R^n)$ and
apply Theorem \ref{thm:thD} to the evolution family $U(t,s)$ generated by the initial-boundary 
value problem \reff{f1l}--\reff{f3l} according to Theorem \ref{thm:th1}. We do this
as explained below. 

Let $\tilde{f} :\R\to C_{per}(\R;\R^n)$ be a bounded and continuous map.
We first show (cf.~Lemma \ref{lem:31}) that a
bounded and continuous map 
$\tilde{u}:\R\to C_{per}(\R;\R^n)$ is a solution to \reff{mild} 
if and only if  $f(x,t):=[\tilde{f}(t)](x)$ and  $u(x,t):=[\tilde{u}(t)](x)$
satisfy equations
\begin{eqnarray}
\lefteqn{
u_j(x,t)=c_j(0,x,t)u_j(1,\tau_j(0,x,t))}\nonumber\\ [2mm] &&
-\int_{0}^{x}d_j(\xi,x,t)\left(\sum_{k\neq j} b_{jk}(\xi,\tau_j(\xi,x,t))u_k(\xi,\tau_j(\xi,x,t))-f_j(\xi,\tau_j(\xi,x,t))\right)d\xi
\label{f10a}
\end{eqnarray} 
for all $j\le n$, where
\beq\label{f8}
c_j(\xi,x,t):=
\exp\left(-\int_{\xi}^{x}\frac{b_{jj}(\eta,\tau_j(\eta,x,t))}{a_j(\eta,\tau_j(\eta,x,t))}d\eta\right), \;\;\;
d_j(\xi,x,t):=
\frac{c_j(\xi,x,t)}{a_j(\xi,\tau_j(\xi,x,t))}.
\ee
Here, for a given $(x,t)\in\R^2$, by $\tau_j(\xi,x,t)$ we denote 
the solution to the initial value problem
\beq
\label{f7l}
\d_{\xi}\tau_{j}(\xi,x,t)=\frac{1}{a_j(\xi,\tau_{j}(\xi,x,t))}, \;\;\; \tau_{j}(x,x,t)=t,
\ee
i.e. $\tau_j(\cdot,x,t)=\xi_j(\cdot,x,t)^{-1}$.

Now, on the account of Theorem \ref{thm:thD}, the  existence of an exponential dichotomy  is reduced 
to the unique solvability of \reff{f10a} for every $f$. To prove the last fact,
we rewrite the system of integral equations \reff{f10a} in the operator form
$$
 u=Cu+Du+Ff
$$
with certain linear bounded operators 
$C,D$ and $F$.
Assumptions  \reff{f5l} and \reff{betacond} of Theorem \ref{thm:th3}
imply that $I-C$ is invertible  (cf.~Lemma \ref{lem:iso}), 
while  Assumptions \reff{a--}--\reff{a+-} imply  that
$$
\|D\|<\frac{1}{\|(I-C)^{-1}\|}
$$
(cf.\ Lemma \ref{lem:Dest} and Corollary \ref{cor:1}).
This, in its turn, gives the invertibility of $I-C-D$, as desired.
Assumptions of Theorem \ref{thm:th2} ensure that the operator 
 $I-C-D$ is Fredholm of index zero and that it is injective, what immediately gives the desired
 bijectivity.

\begin{rem}\rm
The well-known relationship between  the exponential dichotomy and the Green's function
(see e.g. the proof of  \cite[Theorem 1.1]{Latn}) 
can be stated as follows.
Suppose that the assumptions of Theorem \ref{thm:th3} or Theorem \ref{thm:th2} are fulfilled.
Let ${U}(t,s)$ be the evolution family on $C_{per}(\R;\R^n)$
generated by the problem \reff{f1l}--\reff{f3l}.
Then for every bounded function $f:\R\to C_{per}(\R;\R^n)$ the equation 
\reff{mild} has a unique  bounded continuous solution
$u :\R\to C_{per}(\R;\R^n)$ given by the Green's formula
$$
u(t)=\int_{-\infty}^\infty G(t,s)f(s)\,ds,
$$
where
$$
 G(t,s)=\left\{
 \begin{array}{ll}
 U(t,0)PU(0,s) &\mbox{for}\ t>s,\\
-U(s,0)(I-P)U(0,t)  &\mbox{for}\ t<s.
\end{array}
\right.
$$
\end{rem}

The paper is organized as follows. Section \ref{sec:expl}  provides examples showing that the
  assumptions of Theorems \ref{thm:th3} and \ref{thm:th2} are essential.
In Section \ref{sec:proofth1} we give a proof of Theorem \ref{thm:th1} about the existence 
of an evolution family. In Section \ref{sec:equiv}
we establish an equivalence between the mild and weak continuous solution concepts.
 Theorems \ref{thm:th3} and \ref{thm:th2} are proved 
in  Sections \ref{sec:proofth3a} and \ref{sec:proofth2}, respectively. 
 Section \ref{sec:openproblems} contains a concluding discussion and  open problems.

\section{Examples}\label{sec:expl} 

\renewcommand{\theequation}{{\thesection}.\arabic{equation}}
\setcounter{equation}{0}

\begin{ex}\rm
Consider the $2\times 2$-hyperbolic system with non-zero constant coefficients
\beq \label{hse1}
\begin{array}{cc}
    \partial_t u_1+a_1\partial_x u_1+ b_{11}u_1+b_{12}u_2=0, &  \\
    \partial_t u_2+a_2\partial_x u_2+ b_{21}u_1+b_{22}u_2=0 & 
\end{array}
\ee
subjected to the periodic conditions in the space variable 
\beq \label{hse2}
u_1(x,t)=u_1(x+1,t), \;\;\; u_2(x,t)=u_2(x+1,t).
\ee

The problem \reff{hse1}--\reff{hse2} has
constant nontrivial solutions (which obviously prevents an exponential dichotomy on $\R$)
iff 
\beq\label{k2}
b_{11}b_{22}-b_{12}b_{21}=0.
\ee
On the other hand, the assumptions \reff{a--}--\reff{a+-} of Theorem \ref{thm:th3}
for the problem \reff{hse1}--\reff{hse2} are equivalent to \reff{n=2}. This implies the  
inequalities
\beq
\label{bcond}
|b_{12}|< |b_{11}| \mbox{ and } |b_{21}| < |b_{22}|,
\ee
contradicting to \reff{k2}. It follows that the assumptions \reff{a--}--\reff{a+-} are essential
for the statement of  Theorem \ref{thm:th3}.

The problem \reff{hse1}--\reff{hse2} has $x$-independent non-constant time-periodic solutions
(what, again, prevents an exponential dichotomy on $\R$) if the ODE system
\beq\label{k7}
u_1'+b_{11}u_1+b_{12}u_2=u_2'+b_{21}u_1+b_{22}u_2=0
\ee
has non-constant periodic solutions.
The characteristic equation corresponding to this system reads 
\beq\label{k8}
\lambda^2-\lambda(b_{11}+b_{22})+(b_{11}b_{22}-b_{12}b_{21})=0
\ee
The system \reff{k7} has non-constant  time-periodic solutions 
iff the equation  \reff{k8}  has nonzero purely complex solutions. 
The latter is true iff  
$$
b_{22}=-b_{11}, \;\;\; b_{11}^2<-b_{12}b_{21},
$$
which again contradicts to \reff{bcond} and, hence,  to the
assumptions \reff{betacond}--\reff{a+-} of Theorem \ref{thm:th3}. 

Note also that the problem \reff{hse1}--\reff{hse2}  does not satisfy the condition \reff{f143l} of Theorem~\ref{thm:th2}.
\end{ex}

\begin{ex} \rm
Recall that a dichotomy system is exponentially stable if the dichotomy projection coincides with the identity operator. We now show that the assumptions of Theorems \ref{thm:th3} and \ref{thm:th2} do not necessarily imply  the exponential stability.

Suppose that $n=2$ and consider the decoupled system
\begin{equation}\label{k9}
 \begin{array}{cc}
\displaystyle\partial_tu_1+a_1(x,t)\partial_xu_1+b_{11}(x,t)u_1=0, & \\ [2mm]
 \displaystyle\partial_tu_2+a_2(x,t)\partial_xu_2+b_{22}(x,t)u_2=0  &
\end{array}
\end{equation}
with the conditions \reff{f2l} and \reff{f3l}. Suppose that, in addition to the conditions \reff{f5l} and \reff{betacond}, we have $b_{11}>0$ and $b_{22}<0$. The solution to \reff{k9}, \reff{f2l}, \reff{f3l} is given by the formulas
$$
u_{1}(x,t)=\exp{\left(-\int_{s}^{t}b_{11}(\xi_1(r,s,t),r)dr\right)u_{1}^{s}(\xi_{1}(s,x,t))},
$$
$$
u_{2}(x,t)=\exp{\left(-\int_{s}^{t}b_{22}(\xi_2(r,s,t),r)dr\right)u_{2}^{s}(\xi_{2}(s,x,t))}.
$$
It follows  that $u_1$ exponentially decays as $t\to\infty$, while $u_2$ exponentially decays as $t\to-\infty$, for any $u^s\in C_{per}(\R;\R^2)$. One can easily define the dichotomy projection as $Pu=(u_1,0)$, hence $(I-P)u=(0,u_2)$. Since $P\ne I$, the problem is not  exponentially stable. 
\end{ex}

\section{Proof of Theorem \ref{thm:th1}}\label{sec:proofth1}
\setcounter{equation}{0}

Assuming  that the the condition \reff{f5l} is fulfilled, we have to prove that 
the problem \reff{f1l}--\reff{f3l} generates an exponentially bounded evolution family 
$U(t,s)$ on $C_{per}(\R;\R^n)$. For the proof we use \cite[Theorem 2.1]{Km}
stating that under  the zero-order compatibility conditions between \reff{f2l} and \reff{f3l},
which are automatically fulfilled for $u^{s}\in C_{per}(\R;\R^n)$,  the system \reff{f8l}
has  a unique continuous solution. This means that 
there exists a unique strongly continuous evolution family 
$U(t,s)$
on $C_{per}(\R;\R^n)$ associated to \reff{f1l}--\reff{f3l}.
To prove that $U(t,s)$ is exponentially bounded, we use the following a priory estimate derived in the
 proof of \cite[Theorem 2.1]{Km}:
\begin{equation}\label{es1}
\max_{j,x}\max_{s\le\tau\le t}|u_{j}|
\le (3+2n)^{\frac{t-s}{\theta}}\|u^{s}\|\quad \mbox{ for all } t\ge s,
\end{equation}
where 
$$
\theta=
\min\left\{\left(2\sup_{j,x,t}|a_j|\right)^{-1},\left(2n(n+1)\sup_{j,k,x,t}|b_{jk}|\right)^{-1}\right\}.
$$
Since $u=U(t,s)u^{s}$, then from \reff{es1} we get
$$
\|U(t,s)\|\le(3+2n)^{\frac{t-s}{\theta}}\le\exp{\left\{\frac{\log{(3+2n)}}{\theta}(t-s)\right\}}.
$$
This means that the estimate \reff{k3} is true with $K=1$ and $\om=\theta^{-1}\log{(3+2n)}$.
Note that here  we essentially use the boundedness of 
$a_j$ and $b_{jk}$.

 Theorem \ref{thm:th1} is therewith proved.

\section{Existence of an exponential dichotomy on $\R$}
\label{sec:ED} 
\setcounter{equation}{0}

\subsection{Equivalence of the mild and weak solution concepts}
\label{sec:equiv}

Here we establish the equivalence  between  the mild and weak continuous solution concepts, i.e. the equivalence of the equations \reff{mild} and \reff{f10a}, respectively.

Let $BC(\R;C_{per}(\R;\R^n))$ be the Banach space of all bounded and continuous maps 
$u:\R \to C_{per}(\R;\R^n)$, with the norm
$$
\|u\|_\infty:=\sup_{t \in \R}\|u(t)\|,
$$
where $\|\cdot\|$ is the norm in $C_{per}(\R;\R^n)$ introduced in \reff{norm}.
As usual, we identify functions $u \in BC(\R;C_{per}(\R;\R^n))$ with functions  $\tilde{u} \in BC_{per}(\R^2;\R^n)$
by means of $\tilde{u}(x,t)=[u(t)](x)$. Below we will use the same notation for the corresponding elements 
of the two spaces.

To shorten notation, we will write $\tau_j(\xi)=\tau_j(\xi,x,t)$ and $\xi_j(\tau)=\xi_j(\tau,x,t)$.

\begin{lemma}\label{lem:31}
For given  $f \in BC(\R;C_{per}(\R;\R^n))$, the function  $u \in BC(\R;C_{per}(\R;\R^n))$
satisfies  \reff{mild}  if and only if the corresponding function $u \in BC_{per}(\R^2;\R^n)$ 
satisfies  \reff{f10a} with  the corresponding function $f \in BC_{per}(\R^2;\R^n)$.
\end{lemma}
\begin{proof}
The proof is divided into two claims.

{\it Claim 1. Let $f \in BC_{per}(\R^2;\R^n)$. A function $u \in BC_{per}(\R^2;\R^n)$ satisfies
the system \reff{f10a} if and only if it satisfies the system
\begin{equation}\label{contdm}
\begin{array}{ll}
 u_j(x,t)=\displaystyle \exp{\left(\int_{t}^{s}b_{jj}(\xi_j(r),r)dr\right)}u_j^{s}(\xi_j(s))   \\
 -\displaystyle\int_{s}^{t}\exp{\left(\int_{t}^{\tau}b_{jj}(\xi_j(r),r)dr\right)}\left[\sum_{k\neq j}b_{jk}(\xi_j(\tau),\tau)u_{k}(\xi_{j}(\tau),\tau)-f_{j}(\xi_{j}(\tau),\tau)\right] d\tau.
\end{array}
\end{equation}
}

To prove Claim 1, note that \reff{f10a} and \reff{contdm} are two weak formulations of
the problem
 \begin{eqnarray}
\displaystyle &&\partial_{t}u_j
 +a_j(x,t)\partial_{x}u_j
 +\sum_{k=1}^{n}b_{jk}(x,t)u_k
 =f_j(x,t),
 \;\;\; (x,t)\in\R^2,\;\;\; j\le n,
 \label{f1lp}\\
\displaystyle  &&u_j(x+1,t)=u_j(x,t),\;\;\;(x,t)\in\R^2,\;\;\; j\le n,\label{f2lp}
\end{eqnarray}
both obtained by the integration along characteristic curves. 
Let us prove the sufficiency (the necessity is proved similarly).
Suppose that $u$ satisfies \reff{contdm}. One can easily compute the directional 
distributional derivative:
\begin{eqnarray}
\lefteqn{
(\d_t+a_j(x,t)\d_x)u_j(x,t)=
}\nonumber \\
&&\displaystyle -u_{j}^{s}(\xi_{j}(s))b_{jj}(x,t)
\exp{\int_{t}^{s}b_{jj}(\xi_j(r),r)dr}
\displaystyle
-\sum_{k\neq j}b_{jk}(x,t)u_{k}(x,t)+f_j(x,t)\nonumber \\ &&
\displaystyle
+b_{jj}(x,t)\int_{s}^{t}\exp{\int_{t}^{\tau}b_{jj}(\xi_j(r),r)dr}\left[\sum_{k\neq j}(b_{jk}u_{k})(\xi_{j}(\tau),\tau)-f_{j}(\xi_{j}(\tau),\tau)\right] d\tau\nonumber \\
&&\displaystyle
= -\sum\limits_{k=1}^nb_{jk}(x,t)u_{k}(x,t)+f_j(x,t),\quad j\le n,\nonumber
\end{eqnarray}
the last equality being true due to \reff{contdm}. Here we used that 
$$
(\d_t+a_j(x,t)\d_x)\xi_j(\tau,x,t)=0.
$$
Hence, the function $u$ solves the problem \reff{f1lp}--\reff{f2lp} 
where the differential equations are fulfilled in a distributional sense.
Without destroying the equalities in $\D^\prime$, we  rewrite the system
\reff{f1lp} in the form 
  \begin{equation}
(\partial_{t} +a_j(x,t)\partial_{x})\left(c_j(0,x,t)^{-1}u_j\right)
 +c_j(0,x,t)^{-1}\sum_{k\ne j}b_{jk}(x,t)u_k
 =c_j(0,x,t)^{-1}f_j(x,t),
 \;\;\;  j\le n.
 \label{i1}
 \end{equation}
To prove that $u$ satisfies \reff{f10a} pointwise, we use \reff{i1}
and the constancy theorem of distribution theory claiming
that any distribution on an open set with zero generalized derivatives
is a constant on any connected component of the set. As a consequence, the
function
\beq\label{i2}
G_j(x,t)=c_j(0,x,t)^{-1}\left[u_j(x,t)
 +\int_{0}^{x}d_j(\xi,x,t)\left(\sum_{k\neq j} 
(b_{jk}u_k)(\xi,\tau_j(\xi))-f_j(\xi,\tau_j(\xi))\right) d\xi\right]
\ee
is constant along the characteristic $\tau_j(\xi,x,t)$ for all $j\le n$.
Since $G_j(x,t)$ is a  continuous function, by the periodicity condition
 \reff{f2lp}, we get
\beq\label{i3}
G_j(x,t)=G_j(\xi,\tau_j(\xi))=G_j(0,\tau_j(0))=u_j(0,\tau_j(0))=u_j(1,\tau_j(0)).
\ee
Combining \reff{i2} with \reff{i3}, we obtain \reff{f10a}, completing the proof of Claim~1.

{\it Claim 2. Let $f\in BC_{per}(\R^2;\R^n)$. 
A function $u\in BC_{per}(\R^2;\R^n)$ satisfies 
 \reff{contdm} iff $u\in BC(\R;C_{per}(\R;\R^n))$ satisfies \reff{mild}
with $f\in BC(\R;C_{per}(\R;\R^n))$.}

To prove Claim 2, we introduce a two-parameter strongly continuous, exponentially bounded evolution family $U_{0}(t,s)\in\mathcal{L}(C_{per}(\R;\R^n))$, $t\geq s$, by
\begin{equation}\label{evf}
(U_{0}(t,s)u)(x)=\left[\exp{\left(\int_{t}^{s}b_{jj}(\xi_j(x,t,r),r)dr\right)}u_j(\xi_j(s),s)\right]_{j=1}^{n}.
\end{equation}
Let $B(t):C_{per}(\R;\R^n)\to C_{per}(\R;\R^n)$ be a one-parameter family of linear operators defined by
\begin{equation}\label{Boper}
    [B(t)v](x)=\left[-\sum_{k\neq j}b_{jk}(x,t)v_k(x)\right]_{j=1}^{n}.
\end{equation}
In terms of $U_{0}$ the  system \reff{contdm} with $f\equiv 0$ reads
$$
 u(t)=U_{0}(t,s)u(s)+\int_{s}^{t}U_{0}(t,\tau)B(\tau)u(\tau)\,d\tau.
$$
By the definition of the evolution operator $U(t,s)$ we have
$$
 U(t,s)u(s)=U_{0}(t,s)u(s)+\int_{s}^{t}U_{0}(t,\tau)B(\tau)U(\tau,s)u(s)\,d\tau,
$$
which gives us the  following relation between $U$  and $U_0$:
\begin{equation}\label{relf}
    U(t,s)=U_{0}(t,s)+\int_{s}^{t}U_{0}(t,\tau)B(\tau)U(\tau,s)\,d\tau.
\end{equation}

To prove the  sufficiency part of Claim 2, assume that 
 $u\in BC(\R;C_{per}(\R;\R^n))$ satisfies \reff{mild}. By \reff{relf},
the equation \reff{mild} can be written as
\begin{equation}
\begin{array}{cc}\label{lpo}
u(t)=\displaystyle U(t,s)u(s)+\int_{s}^{t} U(t,\tau)f(\tau)d\tau
 &  \\ [2mm]
=\displaystyle U_{0}(t,s)u(s)+\int_{s}^{t}U_{0}(t,\tau)B(\tau)U(\tau,s)u(s)d\tau
 &  \\ [2mm]
+\displaystyle \int_{s}^{t}U_{0}(t,\tau)f(\tau) d\tau+\int_{s}^{t}\int_{\tau}^{t}U_{0}(t,\sigma)B(\sigma)U(\sigma,\tau)f(\tau)\,d\sigma d\tau. &
\end{array}
\end{equation}
Plugging \reff{mild} into the second summand in the right-hand side of \reff{lpo} yields
\begin{equation}\label{starnine}
\begin{array}{cc}
u(t)=\displaystyle U_{0}(t,s)u(s)+\int_{s}^{t}U_{0}(t,\tau)B(\tau)u(\tau)\,d\tau-
\int_{s}^{t}\int^{\tau}_{s}U_{0}(t,\tau)B(\tau)U(t,\sigma)f(\sigma)\,d\sigma d\tau &  \\
+\displaystyle\int_{s}^{t}U_{0}(t,\tau)f(\tau)d\tau+\int_{s}^{t}\int_{\tau}^{t}U_{0}(t,\sigma)B(\sigma)U(\sigma,\tau)f(\tau)\,d\sigma d\tau. & 
\end{array}
\end{equation}
After the  changing of the order of integration in the last summand,
 the third and the last summands cancel out, and we get
\begin{equation}\label{conex}
 u(t)=U_{0}(t,s)u(s)+\int_{s}^{t}U_{0}(t,\tau)B(\tau)u(\tau)\, d\tau
 +\int_{s}^{t}U_{0}(t,\tau)f(\tau)\,d\tau.
\end{equation}
Taking into account \reff{evf} and \reff{Boper}, we see that \reff{conex} 
coincides with \reff{contdm}, as desired.

To prove the  necessity, let $u\in BC_{per}(\R^2;\R^n)$ satisfy \reff{contdm}. 
In terms of \reff{evf} and \reff{Boper}, the equation \reff{contdm} 
coincides with \reff{conex} or, the same, with \reff{starnine}.
Applying the formula \reff{relf} to \reff{starnine}, we get
\begin{equation}\label{cxcx}
\begin{array}{rcl}
\displaystyle u(t)&=&\displaystyle U_{0}(t,s)u(s)+\int_{s}^{t}U_{0}(t,\tau)B(\tau)u(\tau)\,d\tau
+\int_{s}^{t}U_{0}(t,\tau)f(\tau)\,d\tau
\\
\displaystyle &=&\displaystyle U_{0}(t,s)u(s)+\int_{s}^{t}U_{0}(t,\tau)B(\tau)\left[
u(\tau)-\int_{s}^{\tau}U(\tau,\sigma)f(\sigma)\,d\sigma
\right]\,d\tau
\\
\displaystyle &&\displaystyle +\int_{s}^{t}U_{0}(t,\tau)B(\tau)\int_{s}^{\tau}U(\tau,\sigma)f(\sigma)\,d\sigma d\tau
+\int_{s}^{t}U_{0}(t,\tau)f(\tau)\,d\tau
\\
\displaystyle &=&\displaystyle U_{0}(t,s)u(s)+\int_{s}^{t}U_{0}(t,\tau)B(\tau)\left[u(\tau)-\int_{s}^{\tau}U(\tau,\rho)f(\rho)\,d\rho\right]\,d\tau\\
 &&\displaystyle +\int_{s}^{t}U(t,\tau)f(\tau)\,d\tau.
\end{array}
\end{equation}
Set 
\begin{equation}\label{starten}
v(\tau)=u(\tau)-\int_{s}^{\tau}U(\tau,\rho)f(\rho)d\rho.
\end{equation}
Now \reff{cxcx} reads
\begin{equation}\label{vts}
v(t)=U_{0}(t,s)u(s)+\int_{s}^{t}U_{0}(t,\tau)B(\tau)v(\tau)d\tau.
\end{equation}
Combining \reff{vts} with \reff{relf}, we conclude that $v(t)=U(t,s)u(s)$. 
On the account of
\reff{starten}, the equation \reff{cxcx} admits the representation
$$
\begin{array}{cc}
u(t)=\displaystyle U_{0}(t,s)u(s)+\int_{s}^{t}U_{0}(t,\tau)B(\tau)U(\tau,s)u(s)d\tau+\int_{s}^{t}U(t,\tau)f(\tau)d\tau
&  \\
=\displaystyle U(t,s)u(s)+\int_{s}^{t}U(t,\tau)f(\tau)d\tau,
& 
\end{array}
$$
what finishes the proof of Claim~2.
\end{proof}
Lemma \ref{lem:31} readily follows from Claims~1 and~2.

\subsection{Proof of Theorem \ref{thm:th3}}\label{sec:proofth3a}

Assuming that the assumptions  
\reff{f5l}, \reff{betacond} and  \reff{a--}--\reff{a+-}
are fulfilled, we have to prove  that the evolution family of the original problem 
has an exponential dichotomy
on $\R$.  

Let us introduce  operators 
$C,D,F\in \LL(BC_{per}(\R^2;\R^n))$ by 
\begin{eqnarray*}
(Cu)_j(x,t)&:=&\exp\left(-\int_{0}^{x}\frac{b_{jj}(\eta,\tau_j(\eta,x,t))}{a_j(\eta,\tau_j(\eta,x,t))}d\eta\right)u_j(1,\tau_j(0,x,t)),\\
(Du)_j(x,t)&:=&
 -\int_{0}^{x}d_j(\xi,x,t)\sum_{k\neq j} b_{jk}(\xi,\tau_j(\xi,x,t))u_k(\xi,\tau_j(\xi,x,t)) d\xi,\\
(Ff)_j(x,t)&:=&\int_{0}^{x}d_j(\xi,x,t)f_j(\xi,\tau_j(\xi,x,t)) d\xi.
\end{eqnarray*}
Then the equation \reff{f10a} reads
$$
u=Cu+Du+Ff.
$$ 
On the account of Theorem \ref{thm:thD}  and Lemma \ref{lem:31}, we are reduced to show that,
given $f\in BC_{per}(\R^2;\R^n)$, the system \reff{f10a} has  a unique solution in 
$BC_{per}(\R^2;\R^n)$. In other words, we have to prove that the operator 
$I-C-D\in\LL(BC_{per}(\R^2;\R^n))$ is bijective. 

The proof will be divided into two lemmas.

\begin{lemma}\label{lem:iso}
Suppose that the conditions \reff{f5l} and \reff{betacond} are fulfilled. Then 
 the operator $I-C$ is bijective, and for all  $j\le n$ and $\|u\|=1$ it holds
\begin{eqnarray}
\label{--}
\left\|\left[(I-C)^{-1}u\right]_j\right\| \le  \frac{1}{1-e^{-\beta_j^-/\al_j^+}}  &\mbox{if} & \al_j^->0, \;\beta_j^->0,\\
\label{++}
\left\|\left[(I-C)^{-1}u\right]_j\right\| \le  \frac{1}{1-e^{-\beta_j^+/\al_j^-}}  &\mbox{if} & \al_j^+<0, \;\beta_j^+<0,\\
\label{-+}
\left\|\left[(I-C)^{-1}u\right]_j\right\| \le  \frac{e^{\beta_j^+/\al_j^+-\beta_j^-/\al_j^-}}{1-e^{\beta_j^+/\al_j^+}}+1
&\mbox{if} & \al_j^->0, \;\beta_j^+<0,\\
\label{+-}
\left\|\left[(I-C)^{-1}u\right]_j\right\| \le  \frac{e^{\beta_j^-/\al_j^--\beta_j^+/\al_j^+}}{1-e^{\beta_j^-/\al_j^-}}+1
&\mbox{if} & \al_j^+<0, \;\beta_j^->0.
\end{eqnarray}
\end{lemma}
\begin{proof} 
For all $j\le n$ and $(x,t) \in \R^2$ we have
\begin{eqnarray*}
-\frac{b_{jj}(x,t)}{a_j(x,t)}\le -\frac{\beta_j^-}{\al_j^+}<0 &\mbox{if} & \al_j^->0, \;\beta_j^->0,\\
-\frac{b_{jj}(x,t)}{a_j(x,t)}\le -\frac{\beta_j^+}{\al_j^-}<0 &\mbox{if} & \al_j^+<0, \;\beta_j^+<0,\\
0<-\frac{b_{jj}(x,t)}{a_j(x,t)}\le -\frac{\beta_j^-}{\al_j^-} &\mbox{if} & \al_j^->0, \;\beta_j^+<0,\\
0<-\frac{b_{jj}(x,t)}{a_j(x,t)}\le -\frac{\beta_j^+}{\al_j^+} &\mbox{if} & \al_j^+<0, \;\beta_j^->0.
\end{eqnarray*}
Hence,  for all  $j\le n$ and $\|u\|=1$ it holds
\begin{eqnarray}
\label{--a}
\left\|[Cu]_j\right\| \le e^{-\beta_j^-/\al_j^+}<1 &\mbox{if} & \al_j^->0, \;\beta_j^->0,\\
\label{++a}
\left\|[Cu]_j\right\| \le e^{-\beta_j^+/\al_j^-}<1 &\mbox{if} & \al_j^+<0, \;\beta_j^+<0,\\
\label{-+a}
\left\|[Cu]_j\right\| \le e^{-\beta_j^-/\al_j^-} &\mbox{if} & \al_j^->0, \;\beta_j^+<0,\\
\label{+-a}
\left\|[Cu]_j\right\| \le e^{-\beta_j^+/\al_j^+} &\mbox{if} & \al_j^+<0, \;\beta_j^->0.
\end{eqnarray}
Now the bounds  \reff{--} and \reff{++} easily follow from \reff{--a} and \reff{++a}, respectively.

To prove \reff{-+} and \reff{+-}, for an arbitrary fixed  $f \in BC_{per}(\R^2;\R^n)$ consider the equation
 $u=Cu+f$ with respect to $u \in BC_{per}(\R^2;\R^n)$. Then for all $j\le n$ and $ x,t \in \R$ we have \beq
\label{Ceq1}
u_j(x,t)=\exp\left(-\int_{0}^{x}\frac{b_{jj}(\eta,\tau_j(\eta,x,t))}{a_j(\eta,\tau_j(\eta,x,t))}d\eta\right)u_j(1,\tau_j(0,x,t))+f_j(x,t).
\ee
In particular,
\beq
\label{Ceq2}
u_j(1,t)=\exp\left(-\int_{0}^{1}\frac{b_{jj}(\eta,\tau_j(\eta,1,t))}{a_j(\eta,\tau_j(\eta,1,t))}d\eta\right)u_j(1,\tau_j(0,1,t))+f_j(1,t).
\ee
Introduce operators $\tilde{C}_j \in \LL(BC(\R))$ by  
$$
(\tilde{C}_jv)(t):=\exp\left(-\int_{0}^{1}\frac{b_{jj}(\eta,\tau_j(\eta,1,t))}{a_j(\eta,\tau_j(\eta,1,t))}d\eta\right)v(\tau_j(0,1,t)).
$$
We see at once that the operators $\tilde{C}_j$ are bijective and  
\begin{eqnarray}
\label{-+b}
\|\tilde{C}_j^{-1}\| \le e^{\beta_j^+/\al_j^+}<1 &\mbox{if} & \al_j^->0, \;\beta_j^+<0,\\
\label{+-b}
\|\tilde{C}_j^{-1}\| \le e^{\beta_j^-/\al_j^-}<1 &\mbox{if} & \al_j^+<0, \;\beta_j^->0.
\end{eqnarray}
Hence, \reff{Ceq2} is uniquely solvable with respect to $\tilde u(t)=u(1,t)$. Moreover, 
$$
\tilde u_j=-(I-\tilde{C}_j^{-1})^{-1}\tilde{C}_j^{-1}\tilde{f}_j,
$$
where 
$\tilde{f}_j=f_j(1,t)$. Inserting this  into  \reff{Ceq1} and letting $C_ju_j:=(Cu)_j$, we get
$$
u_j=-C_j(I-\tilde{C}_j^{-1})^{-1}\tilde{C}_j^{-1}\tilde{f}_j+f_j,
$$
what entails that
$$
\|(I-C_j)^{-1}\|\le \|C_j\| \|(I-\tilde{C}_j^{-1})^{-1}\| \|\tilde{C}_j^{-1}\|+1.
$$
Now, \reff{-+} and \reff{+-} follow from \reff{-+a}, \reff{+-a}, \reff{-+b} and  \reff{+-b}. 
\end{proof}

By Lemma \ref{lem:iso}, the bijectivity of  $I-C-D\in\LL(BC_{per}(\R^2;\R^n))$
is equivalent to the bijectivity of  $I-(I-C)^{-1}D\in\LL(BC_{per}(\R^2;\R^n))$.

\begin{lemma}\label{lem:Dest} Suppose that the conditions \reff{f5l} and \reff{betacond}
are fulfilled.  
Then
for all  $j\le n$ and $\|u\|=1$ it holds
\begin{eqnarray*}
\label{--c}
\Bigl\|\left[(I-C)^{-1}Du\right]_j\Bigr\| \le  \frac{\beta_j}{\beta_j^-}\frac{\al_j^+}{\al_j^-}  &\mbox{if} & \al_j^->0, \;\beta_j^->0,\\
\label{++c}
\left\|\left[(I-C)^{-1}Du\right]_j\right\| \le  -\frac{\beta_j}{\beta_j^+} \frac{\al_j^-\left(1-e^{-\beta_j^+/\al_j^+}\right)}{\al_j^+\left(1-e^{-\beta_j^+/\al_j^-}\right)}
&\mbox{if} & \al_j^+<0, \;\beta_j^+<0,\\
\label{-+c}
\left\|\left[(I-C)^{-1}Du\right]_j\right\| \le \frac{\beta_j}{\beta_j^-}
\frac{1-e^{-\beta_j^-/\al_j^-}}{1-e^{-\beta_j^+/\al_j^+}}
\left(e^{\beta_j^+/\al_j^+-\beta_j^-/\al_j^-}-e^{-\beta_j^+/\al_j^+}+1\right)
&\mbox{if} & \al_j^->0, \;\beta_j^+<0,\\
\label{+-c}
\left\|\left[(I-C)^{-1}Du\right]_j\right\| \le -\frac{\beta_j}{\beta_j^-}
\frac{1-e^{-\beta_j^+/\al_j^+}}{1-e^{-\beta_j^-/\al_j^-}}
\left(e^{\beta_j^-/\al_j^--\beta_j^+/\al_j^+}-e^{-\beta_j^-/\al_j^-}+1\right)
&\mbox{if} & \al_j^+<0, \;\beta_j^->0.
\end{eqnarray*}
\end{lemma}

\begin{proof}
We see at once that
$$
|(Du)_j(x,t)|\le \beta
 \int_{0}^{x}\frac{1}{|a_j(\xi,\tau_j(\xi,x,t))|}
\exp\left(-\int_{\xi}^{x}\frac{b_{jj}(\eta,\tau_j(\eta,x,t))}{a_j(\eta,\tau_j(\eta,x,t))}d\eta\right)d\xi.
$$
Moreover, for $x \ge \xi$ it holds
\begin{eqnarray*}
\exp\left(-\int_{\xi}^{x}\frac{b_{jj}(\eta,\tau_j(\eta,x,t))}{a_j(\eta,\tau_j(\eta,x,t))}d\eta\right)\le 
 \exp\left(-\frac{\beta_j^-}{\al_j^+}(x-\xi)\right)
&\mbox{if} & \al_j^->0, \;\beta_j^->0,\\
\exp\left(-\int_{\xi}^{x}\frac{b_{jj}(\eta,\tau_j(\eta,x,t))}{a_j(\eta,\tau_j(\eta,x,t))}d\eta\right)\le 
 \exp\left(-\frac{\beta_j^+}{\al_j^-}(x-\xi)\right)
&\mbox{if} & \al_j^+<0, \;\beta_j^+<0,\\
\exp\left(-\int_{\xi}^{x}\frac{b_{jj}(\eta,\tau_j(\eta,x,t))}{a_j(\eta,\tau_j(\eta,x,t))}d\eta\right)\le 
 \exp\left(-\frac{\beta_j^-}{\al_j^-}(x-\xi)\right)
&\mbox{if} & \al_j^->0, \;\beta_j^+<0,\\
\exp\left(-\int_{\xi}^{x}\frac{b_{jj}(\eta,\tau_j(\eta,x,t))}{a_j(\eta,\tau_j(\eta,x,t))}d\eta\right)\le 
 \exp\left(-\frac{\beta_j^+}{\al_j^+}(x-\xi)\right)
&\mbox{if} & \al_j^+<0, \;\beta_j^->0.
\end{eqnarray*}
Therefore,
\begin{eqnarray*}
|(Du)_j(x,t)|\le \frac{\beta_j}{\al_j^-}\int_0^xe^{-\beta_j^-(x-\xi)/\al_j^+}d\xi
=\frac{\beta_j}{\beta_j^-}\frac{\al_j^+}{\al_j^-}\left(1-e^{-\beta_j^-/\al_j^+}\right)
\mbox{ if }  \al_j^->0, \;\beta_j^->0,\\
|(Du)_j(x,t)|\le -\frac{\beta_j}{\al_j^+}\int_0^xe^{-\beta_j^+(x-\xi)/\al_j^-}d\xi
=-\frac{\beta_j}{\beta_j^+}\frac{\al_j^-}{\al_j^+}\left(1-e^{-\beta_j^+/\al_j^+}\right)
\mbox{ if }  \al_j^+<0, \;\beta_j^+<0,\\
|(Du)_j(x,t)|\le \frac{\beta_j}{\al_j^-}\int_0^xe^{-\beta_j^-(x-\xi)/\al_j^-}d\xi
=\frac{\beta_j}{\beta_j^-}\left(1-e^{-\beta_j^-/\al_j^-}\right)
\mbox{ if }  \al_j^->0, \;\beta_j^+<0,\\
|(Du)_j(x,t)|\le -\frac{\beta_j}{\al_j^+}\int_0^xe^{-\beta_j^+(x-\xi)/\al_j^+}d\xi
=-\frac{\beta_j}{\beta_j^+}\left(1-e^{-\beta_j^+/\al_j^+}\right)
\mbox{ if }  \al_j^+<0, \;\beta_j^->0.
\end{eqnarray*}
Taking into account that $\left[(I-C)^{-1}Du\right]_j=(I-C_j)^{-1}(Du)_j$ and combining
the obtained bounds with  Lemma \ref{lem:iso}, we get the desired assertion.
\end{proof} 

\begin{cor}\label{cor:1}
Under the assumptions  of Theorem  \ref{thm:th3},
$$\|(I-C)^{-1}D\|<1.
$$
\end{cor}
Consequently, the operator $I-C-D\in\LL(BC_{per}(\R^2;\R^n))$ is bijective, what completes the proof 
of Theorem  \ref{thm:th3}.

\subsection{Proof of Theorem \ref{thm:th2}}\label{sec:proofth2}

On the account of Lemma \ref{lem:31}, we have to prove the bijectivity of 
the operator $I-C-D\in\LL(BC_{per}(\R^2;\R^n))$.
This will follow from Theorems \ref{thm:th12} and \ref{thm:thuni} below.

\subsubsection{Fredholm alternative}\label{sec:fred}

Here we prove that the operator $I-C-D: BC_{per}(\R^2;\R^n)\to BC_{per}(\R^2;\R^n)$ 
is Fredholm of index zero.

\begin{thm}\label{thm:th12}
Suppose that the conditions \reff{f5l}, \reff{betacond}, \reff{f143l} and \reff{f9l}
are fulfilled. 
Let $\mathcal{K}$ denote the vector space of all bounded continuous solutions to \reff{f10a} with 
$f\equiv0$. Then

$(i)$ $\dim \mathcal{K}<\infty$ and the vector space of all $f\in BC_{per}(\R^2;\R^n)$ such that there exists a bounded continuous solution to \reff{f10a} is 
a closed subspace of codimension $\dim \mathcal{K}$ in $BC_{per}(\R^2;\R^n)$.

$(ii)$ If $\dim \mathcal{K}=0$, then for any $f\in BC_{per}(\R^2;\R^n)$ there exists a unique 
bounded continuous solution $u$ to \reff{f10a}.
\end{thm}

The proof extends the ideas of  \cite{KRKI,KR3}, where the 
Fredholm alternative is proved for 
time-periodic solutions to boundary value hyperbolic problems. 

 One of the technical tools we intend to employ is a generalized Arzela-Ascoli compactness criteria for unbounded domains, see \cite{Lev}. To formulate it, we need a corresponding notion of equicontinuity.
\begin{defn}
A family $\Phi\subset BC_{per}(\R^2;\R^n)$ is called equicontinuous on $[0,1]\times\R$ if
\begin{itemize}
\item $\Phi$ is equicontinuous on any compact set in $[0,1]\times\R$, and 
\item for any $\varepsilon>0$ there exists $T>0$ such that \begin{equation}\label{qwq}
|u(x^{\prime},t^{\prime})-u(x^{\prime\prime},t^{\prime\prime})|<\varepsilon
\end{equation}
for all $x^{\prime}, x^{\prime\prime}\in [0,1]$, all $t^{\prime}, t^{\prime\prime}\in\R\setminus[-T,T]$, and all $u\in\Phi$.
\end{itemize}
\end{defn}

\begin{thm}\label{thm:th21} \rm (a generalized Arzela-Ascoli theorem)\it 
A family $\Phi\subset BC_{per}(\R^2;\R^n)$ is precompact in $BC_{per}(\R^2;\R^n)$ if and only if 
$\Phi$ is bounded in $BC_{per}(\R^2;\R^n)$ and equicontinuous on $[0,1]\times\R$.
\end{thm}

\begin{proof}{\bf\ref{thm:th12}.}
 Lemma \ref{lem:iso} states that the operator $I-C:BC_{per}(\R^2;\R^n)\to BC_{per}(\R^2;\R^n)$ is bijective.
Then the operator $I-C-D$ is Fredholm of index zero if and only if
\beq \label{f37}
I-(I-C)^{-1}D:BC_{per}(\R^2;\R^n)\to BC_{per}(\R^2;\R^n) \;\textrm{is Fredholm of index zero.}
\ee
Nikolsky's criterion \cite[Theorem XIII.5.2]{KA} says that an operator $I+K$ on a Banach space is Fredholm of index zero whenever $K^2$ is compact. Hence, we are done with \reff{f37} if we
 show that the operator $[(I-C)^{-1}D]^2: BC_{per}(\R^2;\R^n)\to BC_{per}(\R^2;\R^n)$ 
 is compact. As the composition of a compact and a bounded operator is a compact operator, it is enough to show that  
$$
D(I-C)^{-1}D:BC_{per}(\R^2;\R^n)\to BC_{per}(\R^2;\R^n) \;\textrm{is compact.}
$$
Since $D(I-C)^{-1}D=D^2+DC(I-C)^{-1}D$ and $(I-C)^{-1}D$ is bounded, it is sufficient to prove that
\beq \label{f38}
D^2, DC:BC_{per}(\R^2;\R^n)\to BC_{per}(\R^2;\R^n)\;\textrm{are compact.}
\ee

To show \reff{f38}, we use Theorem \ref{thm:th21}.
Given $T>0$, set $\Pi(T)=\{(x,t)\in\R^2\,:\,0\le x\le 1,-T\le t\le T\}$. Fix an arbitrary bounded set $Y\subset BC_{per}(\R^2;\R^n)$. For  \reff{f38} it is sufficient to prove the following two statements:
\begin{equation}\label{plusone}
D^2Y \;\textrm{and}\; DCY\;\textrm{are equicontinuous on } \Pi(T)
\mbox{ for an arbitrary fixed } T>0
\end{equation}
and
\begin{equation}\label{plustwo}
\begin{array}{cc}
\textrm{given}\; \varepsilon>0, \;\textrm{there exists}\; T>0 \;\textrm{such that \reff{qwq}}\; \textrm{is fulfilled for all}\; &  \\ 
x^{\prime}, x^{\prime\prime}\in[0,1],\;\;\; t^{\prime}, t^{\prime\prime}\in\R\setminus[-T,T], \;\;\; u\in D^2Y \;\;\;\textrm{and}\;\;\; u\in DCY. & 
\end{array}
\end{equation}

Let us start with  \reff{plusone}. Denote by $C(\Pi(T))$ 
(respectively, $C^1(\Pi(T))$) the Banach space of continuous (respectively,
 continuously differentiable) vector functions $u$ on $\Pi(T)$ such that $u(0,t)=u(1,t)$.
As $C^{1}(\Pi(T))$ is compactly embedded into $C(\Pi(T))$ 
(due to the Arzela-Ascoli theorem), it is sufficient to show that
\beq \label{f39}
\left\|D^2u|_{\Pi(T)}\right\|_{C^{1}(\Pi(T))}+
\left\|DCu|_{\Pi(T)}\right\|_{C^{1}(\Pi(T))}=O\left(\|u\|
\right) \;\textrm{ for all }\;
 u\in Y.
\ee
It should be noted that  for all  sufficiently large~$T$ the functions 
$D^2u$ and $DCu$ restricted to $\Pi(T)$
depend  only on $u$ 
restricted to $\Pi(2T)$.

We will use 
the following formulas
\beq \label{f22}
 \partial_x\tau_{j}(\xi)=
 -\frac{1}{a_j(x,t)}\exp{{\int_{\xi}^{x}\left (\frac{\partial_2a_j}{a_j^2}\right )(\eta,\tau_j(\eta)) d\eta}},
 \ee
 \beq \label{f23}
 \partial_t\tau_{j}(\xi)=
 \exp{{\int_{\xi}^{x}\left (\frac{\partial_2a_j}{a_j^2}\right )(\eta,\tau_j(\eta)) d\eta}},
 \ee
 being true for all $j\le n$, all $\xi,x\in[0,1]$, and all $t\in\R$.
Here and below by $\d_i$ we denote the partial derivative with respect to the $i$-th argument.
Then for all sufficiently large $T>0$ the partial derivatives 
$\partial_xD^2u$, $\partial_{t}D^2u$, $\partial_{x}DCu$, and $\partial_{t}DCu$  on $\Pi(T)$ 
exist and are 
continuous for all $u\in C^{1}(\Pi(2T))$.  Since $C^{1}(\Pi(2T))$ is dense in $C(\Pi(2T))$, 
then the desired property \reff{f39} will follow from the  bound
\beq \label{f310}
\left\|D^2u|_{\Pi(T)}\right\|_{C^1(\Pi(T))}
+\left\|DCu|_{\Pi(T)}\right\|_{C^1(\Pi(T))}
= O(\|u\|_{C(\Pi(2T))}) \;\textrm{for all}\; u\in C^1(\Pi(2T)).
\ee
This bound is proved similarly to \cite[Lemma 4.2]{KR3}:

We start with the estimate
$$
\left\|D^2u|_{\Pi(T)}\right\|_{C^1(\Pi(T))}=O(\|u\|_{C(\Pi(2T))}) \;\textrm{for all}\; u\in 
C^1(\Pi(2T)).
$$
Given $j\le n$ and $u\in C^{1}(\Pi(2T))$, let us consider the following representation for $(D^2u)_j(x,t)$ obtained after the application of the Fubini's theorem:
\beq \label{f311}
(D^2u)_j(x,t)
=\sum_{k\neq j}\sum_{l\neq k}\int_{0}^x\int_{\eta}^{x} d_{jkl}(\xi,\eta,x,t)b_{jk}(\xi,\tau_j(\xi))u_l(\eta,\tau_k(\eta,\xi,\tau_j(\xi))) d\xi d\eta,
\ee
where 
\beq \label{f311d}
d_{jkl}(\xi,\eta,x,t)=d_j(\xi,x,t)d_{k}(\eta,\xi,\tau_{j}(\xi))b_{kl}(\eta,\tau_{k}(\eta,\xi,\tau_j(\xi))).
\ee
Since 
$$
(\d_t+a_j(x,t)\d_x)\varphi(\tau_j(\xi,x,t))=0
$$
for all $j\le n, \varphi\in C^1(\R), x,\xi\in[0,1]$, and $t\in\R$, one can easily check that
$$
\|[(\d_t+a_j(x,t)\d_x)(D^2u)_j|_{\Pi(T)}]\|_{C(\Pi(T))} = O\left(\|u\|_{C(\Pi(2T))}\right)
\mbox{ for all } j\le n \mbox{ and } u\in C^1(\Pi(2T)).
$$
Hence the estimate 
$\left\|\d_xD^2u|_{\Pi(T)}\right\|_{C(\Pi(T))}= O(\|u\|_{C(\Pi(2T))})$ will follow from the following one:
\beq \label{f31jhr1}
\|\d_tD^2u|_{\Pi(T)}\|_{C(\Pi(T))}= O(\|u\|_{C(\Pi(2T))}).
\ee

We are therefore reduced to prove \reff{f31jhr1}. To this end, we start with the following consequence of \reff{f311}:
\begin{eqnarray*}
\lefteqn{
\d_t[(D^2u)_j(x,t)]}
\nonumber\\ &&
=\displaystyle\sum_{k\neq j}\sum_{l\neq k}\int_{0}^x\int_{\eta}^{x} \frac{d}{dt}\Bigl[ d_{jkl}(\xi,\eta,x,t)b_{jk}(\xi,\tau_j(\xi))\Bigr] u_l(\eta,\tau_k(\eta,\xi,\tau_j(\xi))) d\xi d\eta
\nonumber\\ &&
+\displaystyle\sum_{k\neq j}\sum_{l\neq k}\int_{0}^x\int_{\eta}^{x} d_{jkl}(\xi,\eta,x,t) b_{jk}(\xi,\tau_j(\xi))
\nonumber\\ &&
\times
\d_t\tau_k(\eta,\xi,\tau_j(\xi))\d_t\tau_j(\xi)\d_2u_l(\eta,\tau_k(\eta,\xi,\tau_j(\xi))) d\xi d\eta. \label{f312}
\end{eqnarray*}
Let us transform the second summand. Using \reff{f7l}, \reff{f22}, and \reff{f23}, we get
\begin{eqnarray}
\lefteqn{
\frac{d}{d\xi} u_l(\eta,\tau_k(\eta,\xi,\tau_j(\xi)))} \nonumber \\ &&
=\Bigl[\d_x\tau_k(\eta,\xi,\tau_j(\xi))+\d_t\tau_k(\eta,\xi,\tau_j(\xi))\d_{\xi}\tau_j(\xi)\Bigr] \d_2u_l(\eta,\tau_k(\eta,\xi,\tau_j(\xi))) \label{eqwn}
\\ &&
=\left ( \frac{1}{a_j(\xi,\tau_j(\xi))}-\frac{1}{a_k(\xi,\tau_j(\xi))}\right ) \d_t\tau_k(\eta,\xi,\tau_j(\xi))\d_2u_l(\eta,\tau_k(\eta,\xi,\tau_j(\xi))). \nonumber
\end{eqnarray}
Therefore, 
\begin{eqnarray}
\lefteqn{ 
 b_{jk}(\xi,\tau_j(\xi))\d_t\tau_k(\eta,\xi,\tau_j(\xi))\d_2u_l(\eta,\tau_k(\eta,\xi,\tau_j(\xi)))}
\nonumber \\ &&
 =\displaystyle a_j(\xi,\tau_j(\xi))a_k(\xi,\tau_j(\xi))\tilde{b}_{jk}(\xi,\tau_j(\xi))\frac{d}{d\xi} u_l(\eta,\tau_k(\eta,\xi,\tau_j(\xi))), \label{f313}
\end{eqnarray}
where the functions $\tilde{b}_{jk}\in BC_{per}(\R^2;\R)$ are fixed to satisfy \reff{f9l}. Note that $\tilde{b}_{jk}$ are not uniquely defined by \reff{f9l} for $(x,t)$ with $a_{j}(x,t)=a_{k}(x,t)$. Nevertheless, as it follows from \reff{eqwn}, the right-hand side (and, hence, the left-hand side of \reff{f313}) do not depend on the choice of $\tilde{b}_{jk}$, since $\frac{d}{d\xi}u_{l}(\eta,\tau_{k}(\eta,\xi,\tau_{j}(\xi)))=0$ if $a_{j}(x,t)=a_{k}(x,t)$. 

Write
$$
\tilde{d}_{jkl}(\xi,\eta,x,t)
=d_{jkl}(\xi,\eta,x,t)\d_t\tau_j(\xi)a_k(\xi,\tau_j(\xi))a_j(\xi,\tau_j(\xi))\tilde{b}_{jk}(\xi,\tau_j(\xi)),
$$
where $d_{jkl}$ are introduced by \reff{f311d} and \reff{f8}. Using  \reff{f7l} and \reff{f22}, 
we see that the function $\tilde{d}_{jkl}(\xi,\eta,x,t)$ is $C^1$-smooth
 in $\xi$ due to the regularity assumption 
\reff{f9l}. Similarly,
using \reff{f23}, we see that the functions $d_{jkl}(\xi,\eta,x,t)$ and $b_{jk}(\xi,\tau_j(\xi))$ are $C^1$-smooth in $t$.

By \reff{f313} we have
\begin{eqnarray}
\lefteqn{
(\d_tD^2u)_j(x,t)}\nonumber\\ &&
= \displaystyle \sum_{k\neq j}\sum_{l\neq k}\int_{0}^x\int_{\eta}^{x} \frac{d}{dt} [d_{jkl}(\xi,\eta,x,t)b_{jk}(\xi,\tau_j(\xi))] u_l(\eta,\tau_k(\eta,\xi,\tau_j(\xi))) d\xi d\eta
\nonumber\\ &&
+\displaystyle \sum_{k\neq j}\sum_{l\neq k}\int_{0}^x\int_{\eta}^{x}\tilde{d}_{jkl}(\xi,\eta,x,t)\frac{d}{d\xi} u_l(\eta,\tau_k(\eta,\xi,\tau_j(\xi))) d\xi d\eta
\nonumber\\ &&
=\displaystyle \sum_{k\neq j}\sum_{l\neq k}\int_{0}^x\int_{\eta}^{x} \frac{d}{dt} [d_{jkl}(\xi,\eta,x,t)b_{jk}(\xi,\tau_j(\xi))] u_l(\eta,\tau_k(\eta,\xi,\tau_j(\xi))) d\xi d\eta
\nonumber\\ &&
-\displaystyle \sum_{k\neq j}\sum_{l\neq k}\int_{0}^x\int_{\eta}^{x}\d_{\xi}\tilde{d}_{jkl}(\xi,\eta,x,t)u_l(\eta,\tau_k(\eta,\xi,\tau_j(\xi))) d\xi d\eta
\nonumber\\ &&
+\displaystyle \sum_{k\neq j}\sum_{l\neq k}\int_{0}^x\left [\tilde{d}_{jkl}(\xi,\eta,x,t) u_l(\eta,\tau_k(\eta,\xi,\tau_j(\xi)))\right ]_{\xi=\eta}^{\xi=x} d\eta.
\label{k5}
\end{eqnarray}
The desired estimate \reff{f31jhr1} now easily follows from the assumptions 
\reff{f5l}, \reff{betacond} and \reff{f9l} and the equations \reff{f311} and \reff{k5}.

To finish with \reff{f39}, it remains to show that 
\beq \label{f314}
\|\d_tDCu|_{\Pi(T)}\|_{C(\Pi(T))}= O(\|u\|_{C(\Pi(2T))}) \;\textrm{for all} \; u\in C^1(\Pi(2T)), 
\ee
as the estimate for $\d_xDCu$ is obtained similarly to the case of $\d_xD^2u$. In order to prove \reff{f314}, we consider an arbitrary integral contributing into $DCu$, namely
\beq \label{f315}
\int_{0}^{x} e_{jk}(\xi,x,t)b_{jk}(\xi,\tau_j(\xi))u_k(1,\tau_k(0,\xi,\tau_j(\xi))) d\xi,
\ee
where 
$$
e_{jk}(\xi,x,t)=d_j(\xi,x,t)c_k(0,\xi,\tau_j(\xi))
$$
and $j\le n$ and $k\le n$ are arbitrary fixed. Differentiating \reff{f315} in $t$, we get
\begin{eqnarray}
\lefteqn{
\displaystyle \int_{0}^{x} \frac{d}{dt}\Bigl[e_{jk}(\xi,x,t)b_{jk}(\xi,\tau_j(\xi))\Bigr]u_{k}(1,\tau_k(0,\xi,\tau_j(\xi))) d\xi}
\label{dtDC}
\\ &&
\displaystyle \int_{0}^{x} e_{jk}(\xi,x,t)b_{jk}(\xi,\tau_j(\xi))
\d_t\tau_k(0,\xi,\tau_j(\xi))\d_t\tau_j(\xi)\d_{2}u_k(1,\tau_k(0,\xi,\tau_j(\xi))) d\xi.\nonumber
\end{eqnarray}
Let us estimate the second integral; for the first one the desired estimate is obvious. Similarly to the above, we use \reff{f7l}, \reff{f22}, and \reff{f23} to obtain 
\begin{eqnarray*}
\lefteqn{
\frac{d}{d\xi}u_k(1,\tau_k(0,\xi,\tau_j(\xi)))}
\\ &&
=\Bigl[\d_x\tau_k(0,\xi,\tau_j(\xi))+\d_t\tau_k(0,\xi,\tau_j(\xi))\d_{\xi}\tau_j(\xi)\Bigr]\d_{2}u_{k}(1,\omega_k(0,\xi,\tau_j(\xi)))
\\ &&
=\left ( \frac{1}{a_j(\xi,\tau_j(\xi))}-\frac{1}{a_k(\xi,\tau_j(\xi))}\right ) \d_t\tau_k(0,\xi,\tau_j(\xi))\d_{2}u_{k}(1,\tau_k(0,\xi,\tau_j(\xi))).
\end{eqnarray*}
Taking into account \reff{f9l}, the last expression reads
\begin{eqnarray}
\lefteqn{ 
 b_{jk}(\xi,\tau_j(\xi))\d_t\tau_k(0,\xi,\tau_j(\xi))\d_{2}u_{k}(1,\tau_k(0,\xi,\tau_j(\xi)))}
\nonumber \\ &&
 =\displaystyle a_j(\xi,\tau_j(\xi))a_k(\xi,\tau_j(\xi))\tilde{b}_{jk}(\xi,\tau_j(\xi))\frac{d}{d\xi}u_{k}(1,\tau_k(0,\xi,\tau_j(\xi))).\label{f313009}
\end{eqnarray}
Set
$$
\tilde{e}_{jk}(\xi,x,t)
=e_{jk}(\xi,x,t)\d_t\tau_j(\xi)a_k(\xi,\tau_j(\xi))a_j(\xi,\tau_j(\xi))\tilde{b}_{jk}(\xi,\tau_j(\xi).
$$
Using \reff{f22} and \reff{f313009}, let us transform the second summand in \reff{dtDC} as
\begin{eqnarray}
&\displaystyle \int_{0}^{x} e_{jk}(\xi,x,t)b_{jk}(\xi,\tau_j(\xi))
\d_t\tau_k(0,\xi,\tau_j(\xi))\d_t\tau_j(\xi)\d_{2}u_{k}(1,\tau_k(0,\xi,\tau_j(\xi))) 
d\xi&
\nonumber\\ &\displaystyle
=\int_{0}^{x}\tilde{e}_{jk}(\xi,x,t)\frac{d}{d\xi} u_{k}(1,\omega_k(0,\xi,\tau_j(\xi))) d\xi&
\nonumber\\ &\displaystyle
=\Bigl[\tilde{e}_{jk}(\xi,x,t) u_k(1,\tau_k(0,\xi,\tau_j(\xi)))\Bigr]_{\xi=0}^{\xi=x}&
\nonumber\\ &\displaystyle
-\int_{0}^{x}\d_{\xi}\tilde{e}_{jk}(\xi,x,t)u_k(1,\tau_k(0,\xi,\tau_j(\xi))) d\xi.&
\label{fc22}
\end{eqnarray}
The bound \reff{f314} now easily follows from \reff{dtDC} and \reff{fc22}. This finishes the proof of the bound \reff{f310} and, hence the statement \reff{plusone}.

It remains to prove \reff{plustwo}. Fix an arbitrary $\varepsilon>0$. We have to prove the estimates
\begin{equation}\label{d2e0}
|(D^2u)(x^{'},t^{'})-(D^{2}u)(x^{''},t^{''})|<\varepsilon
\end{equation}
and 
\begin{equation}\label{dce}
|(DCu)(x^{'},t^{'})-(DCu)(x^{''},t^{''})|<\varepsilon
\end{equation}
for all $u\in Y$ and all $x^{'}, x^{''}\in [0,1]$,  $t^{'}, t^{''}\in\R\setminus[-T,T]$ and  some $T>0$.

Let us prove \reff{d2e0}. By \reff{f311}, given $j\le n$ and $u\in Y$, we have
\begin{equation}\label{final}
\begin{array}{cc}
|(D^2u)_j(x^{'},t^{'})-(D^{2}u)_j(x^{''},t^{''})|\le|(D^2u)_j(x^{'},t^{'})|+|(D^{2}u)_j(x^{''},t^{''})| &  \\
=2\displaystyle\max\limits_{j\le n}\max\limits_{x\in[0,1]}\max\limits_{t\in\R\setminus[-T,T]}\left|\sum_{k\neq j}\sum_{l\neq k}\int_{0}^{x}\int_{\eta}^{x} d_{jkl}(\xi,\eta,x,t)b_{jk}(\xi,\tau_j(\xi))u_l(\eta,\tau_k(\eta,\xi,\tau_j(\xi))) d\xi d\eta\right| &  \\
\le L\|u\|\displaystyle\max\limits_{k\neq j,l\ne k}\max\limits_{x,\xi,\eta\in[0,1]}\max\limits_{t\in\R\setminus[-T,T]}|b_{jk}(\xi,\tau_j(\xi))b_{kl}(\eta,\tau_k(\eta,\xi,\tau_j(\xi)))|, & 
\end{array}
\end{equation}
the constant $L$ being dependent on $n$, $a_j$ and $b_{jj}$ but not on $u\in Y$ and $b_{jk}$ with $j\neq k$. Since $\|u\|$ is bounded on $Y$, the desired estimate \reff{d2e0} now straightforwardly follows from the estimate \reff{final}, the assumption \reff{f143l}, and the fact that $\tau_j(\xi,x,t)\to\infty$ as $t\to\pm\infty$.

The estimate \reff{dce} is obtained by the same argument, what finishes the proof of \reff{plustwo}. The theorem is proved.
\end{proof}

\subsubsection{Uniqueness of a bounded continuous solution}\label{sec:uniq}

By the assumption \reff{f5l},  there exists  
$m \in \{0,1,\ldots,n\}$ such that for all $(x,t) \in \R^2$
$$
a_j(x,t) > 0 \mbox{ for } j\le m \quad \mbox{ and } \quad
a_j(x,t) < 0 \mbox{ for } m<j\le n.
$$
Write
$$
 x_j:=\left\{
 \begin{array}{rl}
 0 &\mbox{if}\ 1\le j\le m,\\
 1 &\mbox{if}\ m< j\le n.
\end{array}
\right.
$$

\begin{thm}\label{thm:thuni}
Suppose that the conditions 
\reff{f5l} and \reff{f143l} are fulfilled. Moreover, assume that there is $T>0$ such that either
\begin{equation}\label{starone}
\inf_{t<-T}{\int_{1-x_j}^{x_j}\left(\frac{b_{jj}}{a_j}\right)(\eta,\tau_j(\eta,1-x_j,t))d\eta}<0 \;\textrm{ for all }
j\le n
\end{equation}
or
\begin{equation}\label{startwo}
\sup_{t>T}{\int_{1-x_j}^{x_j}\left(\frac{b_{jj}}{a_j}\right)(\eta,\tau_j(\eta,1-x_j,t))d\eta}>0 \;\textrm{ for all }
j\le n.
\end{equation}
Then a bounded continuous solution to \reff{f10a} (if any) is unique.
\end{thm}
\begin{proof}
Given  $T\in\R$, let $Q^{T}=\{(x,t)\in\R^2\,:\,0\le x\le 1,-\infty<t\le T\}$ and $Q_{T}
=\{(x,t)\in\R^2\,:\,0\le x\le 1,T\le t<\infty\}$. 

First assume that the condition \reff{starone} is fulfilled.
By technical reasons, we will use an integral representation of the problem \reff{f1l}--\reff{f2l}, which differs from  \reff{f10a}, namely
\begin{eqnarray}
\lefteqn{
u_j(x,t)=c_j(x_j,x,t)u_j(1-x_j,\tau_j(x_j,x,t))}\nonumber\\ [2mm] &&
-\int_{x_j}^{x}d_j(\xi,x,t)\left(\sum_{k\neq j} b_{jk}(\xi,\tau_j(\xi,x,t))u_k(\xi,\tau_j(\xi,x,t))-f_j(\xi,\tau_j(\xi,x,t))\right) d\xi.
\label{f10axj}
\end{eqnarray} 
It is obtained by integration along the characteristic curves in $x$, in the direction of time decreasing. Note that the integral equations \reff{f10a} and \reff{f10axj} are equivalent in the sense that every continuous solution to \reff{f10a} is a continuous solution to \reff{f10axj} and vice versa. The proof 
of this fact
is similar to the proof of Lemma \ref{lem:31} (Claim 1). The equation \reff{f10axj} is more suitable for our purposes, since  the right-hand side of \reff{f10axj} maps $BC_{per}(Q^{-T};\R^n)$ into itself for each $T>0$,
what is not the case for \reff{f10a}. Fix an arbitrary $T>0$ and consider the equation \reff{f10axj} in $Q^{-T}$. The latter can be written in the operator form 
\begin{equation}\label{operrep}
 u=\hat{C}u+\hat{D}u+\hat{F}f
\end{equation}
 with the operators $\hat{C},\hat{D},\hat{F}:BC_{per}(Q^{-T};\R^n)\to BC_{per}(Q^{-T};\R^n)$ given by 
\beq\label{hatC}
(\hat{C}v)_j(x,t)=c_j(x_j,x,t)v_j(1-x_j,\tau_j(x_j)), \;\;\; j\le n,
\ee
$$
 (\hat{D}v)_j(x,t)=
 -\int_{x_j}^{x}d_j(\xi,x,t)\sum_{k\neq j} b_{jk}(\xi,\tau_j(\xi))v_k(\xi,\tau_j(\xi)) d\xi, \;\;\; j\le n,
$$
and
$$
(\hat{F}f)_j(x,t)=
\int_{x_j}^{x}d_j(\xi,x,t)f_j(\xi,\tau_j(\xi)) d\xi, \;\;\; j\le n,
$$
respectively. Taking into account the definition of $\hat{C}$, the notation \reff{f8}, 
and the assumption \reff{starone}, we get
$$
\|\hat{C}\|_{\mathcal{L}(BC_{per}(Q^{-T};\R^n))}\le\exp\left\{\max_{j\le n}\inf_{t<-T}\int_{1-x_j}^{x_j}\left(\frac{b_{jj}}{a_j}\right)(\eta,\tau_j(\eta,1-x_j,t))d\eta\right\}<1.
$$
It follows that the operator $I-\hat{C}: BC_{per}(Q^{-T};\R^n)\to BC_{per}(Q^{-T};\R^n)$ is bijective 
and, hence, the operator equation \reff{operrep} reads
$$
u=(I-\hat{C})^{-1}\hat{D}u+(I-\hat{C})^{-1}\hat{F}f.
$$
Using the assumption \reff{f143l}, fix $T>0$ so large that the norm of the operator $\hat{D}$ is 
so small that
\begin{equation}\label{3starr}
\|(I-\hat{C})^{-1}\hat{D}\|_{\mathcal{L}(BC_{per}(Q^{-T};\R^n))}<1.
\end{equation}
By the Banach fixed-point theorem, there exists a unique function $u\in BC_{per}(Q^{-T};\R^n)$ satisfying \reff{f10axj} in $Q^{-T}$.

 Now, let us  consider \reff{f1l}--\reff{f2l} in the domain $Q_{-T}$ and
 subject it to the initial condition
\begin{equation}\label{453for}
u_{j}|_{t=-T}=u_j^{-T}(x), \; j\le n.
\end{equation}
Here $u^{-T}(x)$ is the continuous solution to the problem  \reff{f1l}--\reff{f2l}
(or, the same, to  \reff{f10axj}) in $Q^{-T}$ at $t=-T$.
Due to the method of characteristics, the unknown $u$ in $Q_{-T}$ is given by 
the formula  \reff{f10axj} if $\tau_j(x_j,x,t)>-T$ and by the formula 
\begin{equation}\label{+}
u_{j}(x,t)=u_j^{-T}(\xi_j(-T,x,t)), \; j\le n,
\end{equation}
otherwise.
The existence and uniqueness of a continuous solution $u$  to the problem  \reff{f1l}, \reff{f2l},
\reff{453for}
(or, the same, to  \reff{f10axj},
\reff{+})  in $Q_{-T}$ is proved in \cite[Theorem 2.1]{Km}. 

We conclude that the system  \reff{f10a} in the strip $[0,1]\times\R$ has a unique 
continuous solution bounded at $-\infty$. 
This entails that a continuous solution to the system \reff{f10a} (if any) is unique.
The proof under the condition \reff{starone} is complete.

To prove the theorem under the condition \reff{startwo},
we again switch to a suitable integral representation which is equivalent to \reff{f10a}.
The operator form of this integral representation in the domain $Q_{T}$ for an arbitrary 
fixed $T>0$ reads
\beq\label{k6}
u=\tilde{C}u+\tilde{D}u+\tilde{F}f,
\ee
where the  operators $\tilde{C}, \tilde{D}, \tilde{F}: BC_{per}(Q_T;\R^n)\to BC_{per}(Q_T;\R^n)$ 
are defined by 
$$
(\tilde{C}v)_j(x,t)=c_j(1-x_j,x,t)v_j(x_j,\tau_j(1-x_j)), \;\;\; j\le n,
$$
$$
 (\tilde{D}v)_j(x,t)=
 -\int_{1-x_j}^{x}d_j(\xi,x,t)\sum_{k\neq j} b_{jk}(\xi,\tau_j(\xi))v_k(\xi,\tau_j(\xi)) d\xi, \;\;\; j\le n,
$$
and
$$
(\tilde{F}f)_j(x,t)=
\int_{1-x_j}^{x}d_j(\xi,x,t)f_j(\xi,\tau_j(\xi)) d\xi, \;\;\; j\le n.
$$
Similarly to the above,  the operator 
$I-\tilde{C}:BC_{per}(Q_T;\R^n)\to BC_{per}(Q_T;\R^n)$ is bijective by the assumption 
\reff{startwo}, and we have
\begin{equation}\label{starfour}
u=(I-\tilde{C})^{-1}\tilde{D}u+(I-\tilde{C})^{-1}\tilde{F}f.
\end{equation}
Due to \reff{f143l}, for sufficiently large $T$ it holds
$$
\|(I-\tilde{C})^{-1}\tilde{D}|_{BC_{per}(Q_{T};\R^n)}\|<1.
$$
This means that the equation \reff{starfour} has a unique continuous solution in $Q_{T}$. Further we consider  \reff{f1l}--\reff{f2l} 
in the domain $Q^{T}$ with the reverse time. The initial condition is posed at $t=T$, namely
$$
u_{j}|_{t=T}=u_{j}^T(x), \; j\le n,
$$
where $u^{T}(x)$ is the continuous solution to the problem  \reff{f1l}--\reff{f2l}
 in $Q_{T}$ at $t=T$.
By \cite[Theorem 2.1]{Km}, this problem  has a unique continuous solution, what completes the proof.
\end{proof}

It follows that, under the conditions of Theorems \ref{thm:th12} and \ref{thm:thuni},
the system \reff{f10a}  has a unique  solution $u\in BC(\R;C_{per}(R;R^n))$.
Moreover, the conditions \reff{k0} and \reff{k00} entail both 
the condition \reff{betacond} and one of the conditions \reff{starone} and
\reff{startwo}. Hence, under the assumptions of Theorem \ref{thm:th2},
all assumptions of Theorems \ref{thm:th12} and \ref{thm:thuni} are satisfied.
Theorem~\ref{thm:th2} is therewith proved.

\section{Discussion and open problems}\label{sec:openproblems}
\setcounter{equation}{0}

\subsection{Other boundary conditions}
Despite the periodicity conditions \reff{f2l} were essentially used in the proof of the main 
results, we believe that they are not necessary for the statements of Theorems \ref{thm:th3} 
and \ref{thm:th2}.
It would be interesting to extend our approach to other types of boundary conditions. For instance,  the boundary conditions of the reflection type, say,
$$
\begin{array}{l}
\displaystyle
u_j(t,0) = \sum\limits_{k=1}^n\left(r_{jk}^{00}(t)u_k(t,0)+r_{jk}^{01}(t)u_k(t,1)\right)
, \quad  1\le j\le m,\\
\displaystyle
u_j(t,1) = \sum\limits_{k=1}^n\left(r_{jk}^{10}(t)u_k(t,0)+r_{jk}^{11}(t)u_k(t,1)\right)
, \quad  m< j\le n,
\end{array}
$$
are of a particular interest due to numerous applications in  semiconductor laser modeling 
 \cite{LiRadRe,RadWu,Sieber}, boundary feedback control theory \cite{bastin,coron,Pavel},
chemotaxis problems \cite{HRL}.

\subsection{Second-order hyperbolic equations}
For the second-order hyperbolic equation without the zero-order term
\beq\label{21}
\partial^2_tu  - a^2(x,t)\partial^2_xu + a_1(x,t)\d_tu + a_2(x,t)\d_xu =0, \quad 
(x,t)\in\R^2,
\ee
with the periodic boundary condition
\beq\label{22}
u(x,t)=u(x+1,t)
\ee
Theorems \ref{thm:th3} and \ref{thm:th2} provide sufficient conditions for the existence of an
exponential trichotomy (see \cite{ElHa}  for the definition). Indeed, in the new unknowns 
\beq\label{u_sol}
u_1=\partial_tu + a(x,t)\partial_xu,\quad u_2=\partial_tu  - a(x,t)\partial_xu
\ee
the problem \reff{21}--\reff{22} reads as follows:
\beq\label{1s}
\begin{array}{ll}
\displaystyle
\partial_tu_1  - a(x,t)\partial_xu_1 + b_{11}(x,t)u_1+ b_{12}(x,t)u_2 =0  \\\displaystyle
\partial_tu_2  + a(x,t)\partial_xu_2 + b_{21}(x,t)u_1+ b_{22}(x,t)u_2 = 0, 
\end{array}
\ee
\beq\label{2s}
u_j(x+1,t) = u_j(x,t),\quad   j=1,2,
\ee
where
$$
\begin{array}{cc}
\displaystyle
b_{11}=\frac{a_1}{2}+\frac{a_2}{2a}+\frac{a\d_xa-\d_ta}{2a},\quad 
b_{12}=\frac{a_1}{2}-\frac{a_2}{2a}+\frac{a\d_xa-\d_ta}{2a},\\\displaystyle
b_{21}=\frac{a_1}{2}+\frac{a_2}{2a}+\frac{a\d_xa+\d_ta}{2a},\quad
b_{22}=\frac{a_1}{2}-\frac{a_2}{2a}-\frac{a\d_xa+\d_ta}{2a}.
\end{array}
$$
One can easily see that the problems \reff{21}--\reff{22} and \reff{1s}--\reff{2s}
are equivalent in the following sense: For any $c\in\R$ the function $u+c$ is a solution to 
\reff{21}--\reff{22} iff the function $(u_1,u_2)$ given by \reff{u_sol} is a solution to 
 \reff{1s}--\reff{2s}. This means that, under the conditions ensuring the existence of
the  exponential dichotomy for \reff{1s}--\reff{2s}, the problem \reff{21}--\reff{22}
has an exponential trichotomy. For instance, the assumptions \reff{a--}-\reff{a+-}
of Theorem \ref{thm:th3} in the case of constant  coefficients read
$$
\begin{array}{ll}
|aa_1-a_2|<aa_1+a_2& \mbox{ if } \quad aa_1+a_2>0,\\
|aa_1+a_2|<a_2-aa_1& \mbox{ if } \quad a_2-aa_1>0,\\\displaystyle
|aa_1-a_2|<2(-aa_1-a_2)\left(\exp^{\frac{-aa_1-a_2}{2a^2}}-1\right)&  \mbox{ if } \quad aa_1+a_2<0,\\
\displaystyle |aa_1+a_2|<2(aa_1-a_2)\left(\exp^{\frac{aa_1-a_2}{2a^2}}-1\right)&  \mbox{ if } \quad  aa_1-a_2>0.
\end{array}
$$

For the general second-order equation (with the zero-order term)
$$
\partial^2_tu  - a^2(x,t)\partial^2_xu + a_1(x,t)\d_tu + a_2(x,t)\d_xu + a_3(x,t)u=0, \quad 
(x,t)\in\R^2
$$
 the first-order system reads
$$
\begin{array}{ll}
\displaystyle
\partial_tu_1  - a(x,t)\partial_xu_1 + b_{11}(x,t)u_1+ b_{12}(x,t)u_2 + a_3(x,t)u=0  
\\\displaystyle
\partial_tu_2  + a(x,t)\partial_xu_2 + b_{21}(x,t)u_1+ b_{22}(x,t)u_2 + a_3(x,t)u= 0
\\\displaystyle
\partial_tu + a(x,t)\partial_xu=u_1.
\end{array}
$$
Note that the last system  fulfills neither the assumptions of  
Theorem \ref{thm:th3} (as $b_{33}=0$, contradicting to \reff{betacond}) 
nor of Theorem \ref{thm:th2}
(as $a_1=a_3$ and $b_{13}\ne 0$, contradicting to \reff{f9l}). 

This shows that, in general,   second-order hyperbolic equations require
different techniques.

\subsection{Robustness of exponential dichotomy}

The question of robustness of an exponential dichotomy (stability property with respect to data perturbations) for hyperbolic PDEs seems to be a challenging open problem.

Nevertheless, 
our Theorems \ref{thm:th3} and  \ref{thm:th2} give the following consequences.

\begin{cor}\label{thm:rob1}
Under the assumptions of Theorem \ref{thm:th3}, the exponential dichotomy is robust with respect to small perturbations of $a_j$ and $b_{jk}$. Specifically, 
 there exists $\eps>0$ such that the exponential dichotomy persists for all continuously differentiable functions 
$\tilde a_j$ and $\tilde b_{jk}$ that are $1$-periodic in $x$, satisfy the inequalities 
\beq\label{f9lko}
\sup_{j,x,t}|a_j-\tilde a_j|<\eps \; \mbox{ and } \; \sup_{j,k,x,t}|b_{jk}-\tilde b_{jk}|<\eps,
\ee
and fulfill the conditions \reff{f5l}, \reff{betacond}--\reff{a+-}
with $\tilde a_j$ and $\tilde b_{jk}$ in place of $a_j$ and $b_{jk}$, respectively.
\end{cor}

\begin{cor}\label{thm:rob2}
Suppose that the conditions \reff{f5l}, \reff{f143l},  and one of the conditions 
\reff{k0} and \reff{k00} are fulfilled.

1. If
\beq\label{f9ll}
a_j\neq a_k \mbox{ for all } 1\le j\neq k\le n \mbox{ and } (x,t)\in\R^2,
\ee 
then the exponential dichotomy is robust under small perturbations of $a_j$ and $b_{jk}$.
Specifically, there exists $\varepsilon>0$ such that the exponential dichotomy persists for all continuously differentiable functions $\tilde{a}_j$ and $b_{jk}$ that are 1-periodic in $x$, satisfy the inequalities \reff{f9lko}, and fulfill the conditions \reff{f5l}, \reff{k0}, \reff{f9ll} with $\tilde{a}_j$ and $\tilde{b}_{jk}$ in place of $a_j$ and $b_{jk}$, respectively.

2. If
$$
\begin{array}{ll}
b_{jk}\equiv 0 
\mbox{ for all } 1 \le j \not= k \le n \\
\mbox{in a neighborhood of the set }
\{(x,t)\in\overline\Omega\,:\, a_k(x,t)=a_j(x,t)\},
\end{array}
$$
then the exponential dichotomy   is robust under small perturbations of $a_j$.
Specifically,  there exists $\varepsilon>0$ such that the exponential dichotomy persists for all continuously differentiable functions $\tilde{a}_j$ that are 1-periodic in $x$, satisfy the first inequality in \reff{f9lko}, and fulfill the conditions \reff{f5l}, 
\reff{k0}, \reff{f143l}
with $\tilde a_j$ in place of $a_j$.
\end{cor}

Henry \cite[Theorem 7.6.10]{Henry} established a general  sufficient condition of the robustness of an exponential dichotomy for abstract evolution equations. 
Attempts to apply this approach to hyperbolic PDEs  meet 
complications caused by loss of smoothness. In \cite{TKM} these complications are overcome
in the case of  boundary conditions of 
the so-called smoothing type when the solutions of initial-boundary value problems 
become more regular than  the initial data after some time. In the general case, the  robustness issue for hyperbolic PDEs remains unexplored.

\section*{Acknowledgments}

Roman Klyuchnyk  was supported by the BMU-MID Erasmus Mundus Action 2 grant
 MID2012B1422. He expresses his gratitude to the Applied Analysis group at the Humboldt
 University of Berlin for its kind hospitality.

\end{document}